\documentclass[11pt,twoside]{amsart}
\usepackage{amsmath, amsthm, amscd, amsfonts, amssymb, graphicx}
\usepackage{enumerate}
\usepackage[colorlinks=true,
linkcolor=blue,
urlcolor=cyan,
citecolor=red]{hyperref}
\usepackage{mathrsfs}
\newtheorem{theorem}{Theorem}[section]
\newtheorem{lemma}{Lemma}[section]
\newtheorem{remark}{Remark}[section]
\newtheorem{definition}{Definition}[section]
\newtheorem{corollary}{Corollary}[section]

\newtheorem{proposition}{Proposition}[section]
\numberwithin{equation}{section}

\begin{document}
\title{From positive to accretive matrices }
\author{Yassine Bedrani, Fuad Kittaneh and Mohammad Sababheh}
\subjclass[2010]{15A02, 15B48, 47A63, 47A64.}
\keywords{matrix monotone function, accretive matrix, Ando's inequality, Choi's inequality, matrix mean} \maketitle

\pagestyle{myheadings}
\markboth{\centerline {}}
{\centerline {}}
\bigskip
\bigskip
\begin{abstract}
The main goal of this paper is to discuss the recent advancements of matrix means from positive matrices to accretive matrices in a more general setting. In particular, we present the general form governing the well established definition of geometric mean, then we define arbitrary matrix means and functional calculus for accretive matrices.

Applications of this new discussion involve generalizations of known inequalities from the setting of positive matrices to that of accretive matrices. This includes the arithmetic-harmonic mean comparisons, monotonicity of matrix means, Ando's inequality, Choi's inequality, Ando-Zhan subadditive inequality and much more.

\end{abstract}
\section{Introduction}
Let $\mathcal{M}_n$ be the algebra of all complex $n\times n$ matrices. We recall some basic definitions related to this algebra. A matrix $A\in\mathcal{M}_n$ is said to be positive semidefinte,  denoted  by $A\geq 0$, if 
$\left<Ax,x\right>\geq 0$ for all  $x\in\mathbb{C}^n.$ If $A\geq 0$ is invertible, it is  called positive definite, and it is denoted by $A>0.$ The class of positive definite matrices will be denoted by $\mathcal{M}_n^{+}.$ For two Hermitian matrices $A,B\in\mathcal{M}_n$, we say that $A\leq B$ (or $A<B$) if $B-A\geq 0$ (or $B-A>0$). The relation $A\leq B$ defines a partial ordering on the class of Hermitian matrices. The identity matrix in $\mathcal{M}_n$ will be denoted by $\mathcal{I}_n$, or $\mathcal{I}$ if there is no confusion. The theory of matrix means for two positive matrices has been developed by Kubo and Ando in \cite {kubo_ando} as follows.

A matrix mean $\sigma$ on $\mathcal{M}_n^+$ is a binary operation $A\sigma B$ satisfying the following requirements: 
\begin{itemize}
\item $A\leq C$ and $B\leq D$ imply $A\sigma B\leq C\sigma D;$ for any $A,B,C,D\in\mathcal{M}_n^+$.
\item $C^{*}(A\sigma B)C= (C^{*}AC)\sigma (C^{*}BC);$ for any $A,B\in\mathcal{M}_n^+$ and any invertible $C\in\mathcal{M}_n$.
\item $A_k\downarrow_k A$ and $B_k\downarrow_k B$ imply $(A_k\sigma B_k)\downarrow_k (A\sigma B)$; for any $A_k,B_k,A,B\in\mathcal{M}_n^+.$
\item $\mathcal{I}\sigma \mathcal{I}=\mathcal{I}.$
\end{itemize}
Standard examples of matrix means are given by \cite{pusz_wor}
\begin{itemize}
\item The weighted  arithmetic mean $A\nabla_{\lambda} B=(1-\lambda)A+\lambda B,$
\item The weighted Harmonic mean $A!_{\lambda}B=((1-\lambda)A^{-1}+\lambda B^{-1})^{-1},$
\item The weighted geometric mean $A\sharp_{\lambda} B=A^{\frac{1}{2}}\left(A^{-\frac{1}{2}}BA^{-\frac{1}{2}}\right)^{\lambda}A^{\frac{1}{2}},$
\end{itemize}
where $A,B\in\mathcal{M}_n^+$ and $0\leq \lambda\leq 1.$ When $\lambda=\frac{1}{2}$, we drop $\lambda$ from the above notations, and we simply write $\nabla,!$ and $\sharp.$\\
For two matrix means $\sigma, \tau$, we say that $\sigma\leq \tau$ if $A\sigma B\leq A\tau B$ for all $A,B\in\mathcal{M}_n^+.$ In particular, we have $!_{\lambda}\leq \sharp_{\lambda}\leq \nabla_{\lambda},$ \cite{ando_1}. That is, if $A,B\in\mathcal{M}_n^{+}$, then
$$A!_{\lambda} B\leq A\sharp_{\lambda} B\leq A  \nabla_{\lambda} B, 0\leq \lambda\leq 1.$$
Other celebrated relations for these means are 
$$A\nabla_{\lambda} B=B\nabla_{1-\lambda}A, A\sharp_{\lambda} B=B\sharp_{1-\lambda}A, A!_{\lambda} B=B!_{1-\lambda}A.$$ 
The theory of matrix means is strongly related to that of matrix monotone functions, where any mean $\sigma$ on $\mathcal{M}_n^{+}$ is characterized by
\begin{align*}
A\sigma B=A^{\frac{1}{2}}f\left(A^{-\frac{1}{2}}BA^{-\frac{1}{2}}\right)A^{\frac{1}{2}},
\end{align*}
for a certain matrix monotone function $f:(0,\infty)\to (0,\infty),$ with $f(1)=1.$ Recall that a function $f:J\to \mathbb{R}$ is said to be matrix monotone if it preserves matrix order. That is, if it satisfies
\begin{align*}
f(A)\leq f(B) \;{\text{whenever}}\;A\leq B,
\end{align*}
for the  Hermitian matrices $A,B$ whose spectra are in the interval $J.$

Theory of matrix means for positive matrices has been well developed and studied in the literature. We refer the reader to \cite{ando_1,kubo_ando,nishio_ando,pusz_wor} as a sample of articles treating this topic.

 A matrix $A\in\mathcal{M}_n$ is said to be accretive if its real part, defined by $\Re A=\frac{A+A^*}{2}$ is positive definite (i.e., $\Re A>0.$) This condition is equivalent to the fact that the numerical range $W(A)$ of $A$ satisfies
$$W(A):=\{\left<Ax,x\right>:x\in\mathbb{C}^n, \|x\|=1\}\subset {\text{the right half complex plane}}.$$ It is readily seen that the class of accretive matrices is a convex cone of $\mathcal{M}_n^+$ that is invariant under inversion.

When studying properties of accretive matrices, it is necessary to recall the definition of sectorial matrices.  For $0\leq \alpha<\frac{\pi}{2},$ we define the sector
$$ S_{\alpha}=\{z \in \mathbb{C}: \Re(z) > 0, |\Im(z)| \leq \tan(\alpha) \Re(z)  \}. $$
A matrix $A$ whose numerical range is a subset of a sector $S_{\alpha}$, for some $\alpha \in [0, \pi/2)$, is called a sectorial matrix. It is clear that a sectorial matrix is necessarily accretive.

Using the integral representation of a matrix monotone function, Hiai \cite{hiai_new} extended the definition of a matrix mean from positive to dissipative matrices. Recall that a matrix $A$ is said to be dissipative if $\Im A>0$. Thus, if $A$ is dissipative, then $-iA$ is accretive, and if $A$ is accretive then $iA$ is dissipative. The definition given in \cite{hiai_new} was motivated by the study of concavity of a certain trace functional, and no further properties were discussed neither in \cite{hiai_new} nor in later literature.

Related to our approach, Drury \cite{drury} extended the definition of geometric mean from the setting of positive definite matrices to the class of accretive matrices, where he defined the geometric mean of two accretive matrices $A,B \in\mathcal{M}_n$ by
\begin{align}\label{drury_def}
A\sharp B=\left( \frac{2}{\pi}\int_{0}^{\infty}(tA+t^{-1}B)^{-1}\frac{dt}{t}\right)^{-1}.
\end{align}
In the same paper, Drury discussed many interesting properties of this geometric mean and, in particular, he showed that when $A,B>0$, his definition coincides with  $A\sharp B=A^{\frac{1}{2}}\left(A^{-\frac{1}{2}}BA^{-\frac{1}{2}}\right)^{1/2}A^{\frac{1}{2}}.$

A little later, Ra\"{i}ssouli et. al. \cite{raissouli} presented the weighted geometric mean for two accretive matrices $A,B\in\mathcal{M}_n$ by the formula
\begin{align}\label{raissouli_def}
A\sharp_\lambda B=	\frac{\sin(\lambda\pi)}{\pi}\int_{0}^{\infty}t^{\lambda-1}(A^{-1}+tB^{-1})^{-1}dt, 0<\lambda<1.
\end{align}
In the same paper, the authors showed that when $\lambda=\frac{1}{2}$, Drury's definition given in \eqref{drury_def} coincides with \eqref{raissouli_def}. Further, they showed that when $A,B>0$, then their definition \eqref{raissouli_def} reduces to
\begin{align*}
A\sharp_\lambda B= A^{\frac{1}{2}}\left(A^{-\frac{1}{2}}BA^{-\frac{1}{2}}\right)^{\lambda}A^{\frac{1}{2}}, 0<\lambda<1;
\end{align*}
which is consistent with the definition of $\sharp_{\lambda}$ when $A,B\in\mathcal{M}_n^+$. We will show that this is also true for accretive matrices.

 We refer the reader to \cite{raissouli}, where many properties for the weighted geometric mean of accretive matrices  have been discussed and matched with the corresponding properties for positive ones.

Our first target in this article is to show that \eqref{raissouli_def} follows from a more general setting for matrix monotone functions and matrix means; which then leads us to a reasonable generalization of the concept of matrix means for accretive matrices. Further, we discuss Ando's and Choi's inequalities for accretive matrices.

For our purpose, we will need the following  notation and preliminaries.

A function $f:J\to\mathbb{R}$ is said to be matrix convex if it is continuous and $f\left((1-t)A+tB\right)\leq (1-t)f(A)+tf(B),$ for $0\leq t\leq 1$ and the Hermitian matrices  $A,B$ with spectra in $J$. If $-f$ is matrix convex, $f$ is called matrix concave. 

We adopt the notation 
$$ \mathfrak{m}=\{f:(0,\infty)\to (0,\infty); f\; {\text{is a matrix monotone function with}} \;f(1)=1\}.$$

Matrix monotone functions and matrix concave functions are strongly related, as follows \cite[Theorem 2.4]{ Uchiyama} and \cite[Theorems 2.1, 2.3, 3.1, 3.7]{ando_hiai}.
\begin{proposition}\label{oper_intro_prop}
Let $f:(0,\infty)\to [0,\infty)$ be continuous. Then 
\begin{enumerate}
\item[(i)] $f$ is matrix monotone decreasing if and only if $f$ is matrix convex and $f(\infty)<\infty$.
\item[(ii)] $f$ is matrix monotone increasing if and only if $f$ is matrix concave.
\end{enumerate}
\end{proposition}

The following characterization of $f\in\mathfrak{m}$ will be useful for our analysis.
\begin{lemma}\label{hansen_lem} (\cite[Theorem 4.9]{hansen}) 
Let $f\in \mathfrak{m}.$ Then
$$f(x)=\int_{0}^{1}(1!_tx)d\nu_f(t),$$ where $\nu_f$ is a probability measure on $[0,1].$
\end{lemma}
 
 Consequently, if $f\in \mathfrak{m}$ and $\sigma_f$ is the corresponding mean (i.e., $A\sigma_f B=A^{\frac{1}{2}}f\left(A^{-\frac{1}{2}}BA^{-\frac{1}{2}}\right)A^{\frac{1}{2}}$), then it is not hard to show that for  $A,B\in\mathcal{M}_n^{+},$ we have
 \begin{align}\label{mean_har_prob_intro}
    A\sigma_fB=\int_{0}^{1}A!_{t}B\;d\nu_f(t),
\end{align} 
where $\nu_f$ is a probability measure on $[0,1];$ depending on $f$.

In our discussion, we will need to deal with $f(z)$ where $z\in \mathbb{C}.$ We first recall the following  celebrated result of L\"{o}wener about matrix monotone functions. 
\begin{lemma}\cite[Theorem V.4.7]{bhatia_matrix}
Let $f\in\mathfrak{m}$. Then $f$ has an analytic continuation to $\mathbb{C}\backslash (-\infty,0]$.
\end{lemma}
Notice that when $f\in\mathfrak{m},$ the integral representation in Lemma \ref{hansen_lem} applies when $x\in (0,\infty).$ If we use $f$ to denote the analytic continuation of $f$ to $\mathbb{C}\backslash (-\infty,0]$, we have the following.
\begin{proposition}\label{prop_our_hansen}
Let $f\in\mathfrak{m}$. Then, for $z\in\mathbb{C}\backslash (-\infty,0]$, the integral representation
$$f(z)=\int_{0}^{1}1!_tz\;d\nu_f(t),$$ holds true, where $\nu_f$ is as in Lemma \ref{hansen_lem}.
\end{proposition}
\begin{proof}
Notice that when $z\not\in (-\infty,0]$, the quantity $1!_tz$ is well defined. For such $z$, define $g(z)=\int_{0}^{1}1!_tz\;d\nu_f(t).$ We show that $f=g.$\\
We  show  that $g$ is analytic in $\mathbb{C}\backslash (-\infty,0].$ Indeed, let $\Gamma$ be any closed circle in $\mathbb{C}\backslash (-\infty,0].$ First, we show that 
\begin{align*}
\int_{\Gamma}\left(\int_{0}^{1}1!_tz\;d\nu_f(t)\right)dz&=\int_{0}^{1}\left(\int_{\Gamma}1!_tz\;dz\right)d\nu_f(t).
\end{align*}

Notice that (by letting $z=a+re^{i\theta}$)
\begin{align*}
\int_{0}^{2\pi}\int_{0}^{1}\left|1!_t(a+re^{i\theta})\right|d\nu_f(t)rd\theta&=r\int_{0}^{1}\int_{0}^{2\pi}\left|1!_t(a+re^{i\theta})\right|d\theta d\nu_f(t).
\end{align*}
But the function $F(t,\theta)=\left|1!_t(a+re^{i\theta})\right|$ is continuous on the compact set $[0,1]\times [0,2\pi]$. Therefore, $s:=\sup_{(t,\theta)}F(t,\theta)<\infty$, and hence
\begin{align*}
\int_{0}^{2\pi}\int_{0}^{1}\left|1!_t(a+re^{i\theta})\right|d\nu_f(t)rd\theta&\leq 2\pi s<\infty.
\end{align*}
This means that
\begin{align*}
\int_{0}^{2\pi}\int_{0}^{1}\left(1!_t(a+re^{i\theta})\right)d\nu_f(t)rd\theta&=r\int_{0}^{1}\int_{0}^{2\pi}\left(1!_t(a+re^{i\theta})\right)d\theta d\nu_f(t),
\end{align*}
which is equivalent to
$$\int_{\Gamma}\left(\int_{0}^{1}1!_tz\;d\nu_f(t)\right)dz=\int_{0}^{1}\left(\int_{\Gamma}1!_tz\;dz\right)d\nu_f(t).$$

Then
\begin{align*}
\int_{\Gamma}g(z)dz&=\int_{\Gamma}\left(\int_{0}^{1}1!_tz\;d\nu_f(t)\right)dz\\
&=\int_{0}^{1}\left(\int_{\Gamma}1!_tz\;dz\right)d\nu_f(t)\\
&=0,
\end{align*}
where we have used the fact that $z\mapsto 1!_tz$ is analytic in $\mathbb{C}\backslash (-\infty,0]$, for every $t\in [0,1].$ Since $\int_{\Gamma}g(z)dz=0$ for any circle, in the domain, it follows that $g$ is analytic in $\mathbb{C}\backslash (-\infty,0]$. Finally, since $f$ and $g$ are analytic functions having the same values in $(0,\infty),$ it follows that $f=g$. This completes the proof.
\end{proof}

Our first result will be to show that \eqref{drury_def} and \eqref{raissouli_def}   follow as an application of the above computations. That is, we show that when $0<\lambda<1$, a probability measure $\nu_{\lambda}$ on $[0,1]$ exists such that for two accretive matrices $A,B$, 
\begin{align*}
\int_{0}^{1}A!_t B\; d\nu_{\lambda}(t)=\frac{\sin (\lambda\pi)}{\pi}\int_{0}^{\infty}t^{\lambda-1}(A^{-1}+tB^{-1})^{-1}dt.
\end{align*}

Once this has been shown, we introduce the definition of matrix means for accretive matrices in the setting of arbitrary matrix monotone functions, then we extend the study in the same theme to the discussion of Ando's and Choi's inequalities. Recall that these inequalities state \cite {Ando,ando_1, choi,kubo_ando}, respectively,
\begin{align}\label{andos_inequality_intro}
\Phi(A\sigma_f B)\leq \Phi(A)\sigma_f \Phi(B),
\end{align}
and
\begin{align}\label{chois_inequality_intro}
\Phi(f(A))\leq f(\Phi(A)),
\end{align}
whenever $\Phi:\mathcal{M}_n\to\mathcal{M}_r$ is a unital positive linear mapping, $A,B\in \mathcal{M}_n$ are positive  and $f:(0,\infty)\to (0,\infty)$ is a matrix monotone function. In this context, recall that a linear mapping $\Phi:\mathcal{M}_n\to\mathcal{M}_r$ is called positive if it preserves positive matrices (i.e., if $\Phi(A)\geq 0$ when $A\geq 0$) and it is called unital if $\Phi(\mathcal{I})=\mathcal{I}.$ In fact the assumption of $\Phi$ being unital is not necessary for
\eqref{andos_inequality_intro} while necessary for \eqref{chois_inequality_intro}

So, we will show the accretive versions of both \eqref{andos_inequality_intro} and \eqref{chois_inequality_intro}. For this to be accomplished, we need to remind the reader of the meaning of $f(A)$, when $A$ is a general matrix.

 Let $f:D\to\mathbb{C}$ be an analytic complex function on the domain $D.$ The Cauchy integral formula assures that for $a\in D,$
\begin{align*}
    f(a)=\frac{1}{2\pi i}\int_{\Gamma}\frac{f(z)}{z-a}dz,
\end{align*}
where $\Gamma$ is a simple closed curve in $D$ that winds once around $a$. Extending this definition to matrices (or operators in general) is made using the Dunford integral
\begin{align}\label{dunford_intro}
    f(A)=\frac{1}{2\pi i}\int_{\Gamma} f(z)(z\mathcal{I}-A)^{-1}dz,
\end{align}
where $\Gamma$ is a simple closed curve in the resolvent of $A$ that winds once around each eigenvalue of $A$ and lies entirely inside $D$. Of course $ D$ contains the spectrum of $A$.

For example, letting $f:\mathbb{C}\backslash (-\infty,0] \to \mathbb{C}$ be $f(z)=z^{\lambda}, 0<\lambda<1,$ we define 
\begin{align}\label{def_frac_power_intro}
    A^{\lambda}=\frac{1}{2\pi i}\int_{\Gamma}z^{\lambda}(z\mathcal{I}-A)^{-1}dz,
\end{align}
 where $\Gamma$ is any closed curve avoiding $(-\infty,0]$ in the resolvent of $A$, so that $\Gamma$ winds once around each eigenvalue of $A.$

So, fractional powers are not only defined for positive matrices. They can be defined for any matrix whose eigenvalues are not in $(-\infty,0]$. 

In the sequel, we prove that for accretive matrices, the above Dunford integral may be replaced by a harmonic-mean integral. This approach will enable us to achieve our target. In fact Theorem \ref{thm_main_1} provides an alternative formula for any matrix with eigenvalues avoiding $(-\infty,0]$.

The organization of this paper is as follows. First, we list some lemmas that we will need in our analysis. Then, the geometric mean for accretive matrices is studied further. Once the geometric mean is settled, we introduce the alternative formula for $f(A)$, when $A$ is accretive, then the definition of arbitrary  matrix mean for accretive matrices will be presented. After that numerous applications that involve generalizations of several results from the setting of positive  to  accretive matrices will be presented.

While this article treats accretive matrices, it will be noticed that the corresponding results for positive ones will be special cases of our results. This means that this article can be  viewed as an exposition for celebrated inequalities of positive matrices.
\section{Some preliminary results}
In this section we list  different results that we will need in the sequel. These results can be found in the stated references. 

Further, one goal of this paper is to extend most of these results from the setting of positive matrices to accretive ones. So, to make it easier for the reader, we will mention the corresponding result from the subsequent sections that extends the stated result. We begin with the following version of the celebrated Jensen inequality.
\begin{lemma}\cite{Furuta} \label{f<>} Let $A\in\mathcal{M}_n^+$. Then for  $f\in\mathfrak{m}$  and any unit vector $x$, 
\begin{align}
\left\langle f(A)x,x\right\rangle \leq f\left( \left\langle Ax,x\right\rangle \right) .
\end{align}
\end{lemma}
We refer the reader to Corollary \ref{theorem_inner_product} below for the extension of this result to accretive or sectorial matrices.
\begin{lemma}\cite{Ando_2} \label{<sigma_f>} Let $A, B\in\mathcal{M}_n^+$. If $f\in \mathfrak{m}$,  then for any unit vector  $x$,
\begin{align}
\left\langle (A\sigma_f B)x,x\right\rangle \leq \left\langle Ax,x\right\rangle \sigma_f\left\langle Bx,x\right\rangle .
\end{align} 
\end{lemma} 
In Corollary \ref{theorem_sigma_inner} below, we present  the extension of this result to accretive or sectorial matrices.\\
Recall that a norm $\|\cdot\|$ on $\mathcal{M}_n$ is said to be unitarily invariant if it satisfies $\|UAV\|=\|A\|$ for any $A,U,V\in\mathcal{M}_n$ such that $U$ and $V$ are unitary matrices.
\begin{lemma}\cite{Ando_2}\label{sigma_norm} Let $A, B\in\mathcal{M}_n^+$. If $f\in \mathfrak{m}$, then for any unitarily invariant norm $ \parallel\cdot\parallel,$
\begin{align}
\parallel A \sigma_f B\parallel\leq \parallel A\parallel\sigma_f\parallel B\parallel.
\end{align}
\end{lemma}
The extension of this result to accretive or sectorial matrices is presented in Theorem \ref{theorem_norm_of_sigma} below.
\begin{lemma}\cite{kubo_ando}\label{AM-GM-HM}
Let $ A, B \in \mathcal{M}_{n}^+ $ and let $f\in \mathfrak{m}$ be such that $f'(1)=t$ for some $t\in(0,1)$. Then
\begin{center}
$ A!_{t}B\ \leq\ A\sigma_f B\ \leq\ A\nabla_{t}B   $
\end{center}
\end{lemma}
This result has its accretive version, which we present in \ref{thm_amgmhm_accretive} below.

The following is a special form of the Choi-Davis inequality for accretive matrices.
\begin{lemma}\cite{Lin 2}\label{RA} 
 If $ A\in \mathcal{M}_{n} $ is accretive, then
 \begin{center}
 $ \Re(A^{-1})\leq (\Re A)^{-1} $
 \end{center}
 \end{lemma}
 We refer the reader to Proposition \ref{prop_f_r_f}, where an analogue of this result is given about matrix concave functions.
\begin{lemma}\cite{Drury1} \label{sec}
 If $ A\in \mathcal{M}_{n} $ with $ W(A)\subset S_{\alpha} $, then
 \begin{center}
 $\sec^{2}(\alpha)\hspace{0.25cm} \Re(A^{-1})\geq (\Re A)^{-1} $
 \end{center}
\end{lemma}
In Proposition \ref{prop_f_real_sec_f}, we give an analogue of this result for matrix concave functions.
\begin{lemma}\label{2}\cite{raissouli}
Let $ A, B\in\mathcal{M}_n $ be accretive matrices and $ 0<t<1 $. Then 
 \begin{equation}
 \Re(A!_{t}B)\geq (\Re A)!_{t}(\Re B).
 \end{equation}
\end{lemma}
It is interesting to investigate this result for an arbitrary mean $\sigma_f$. This will be done in Proposition \ref{r_a_sigma_b>} below.
\begin{lemma}(\cite{Lin 1}) \label{lemma_reverse_real_!t} Let $ A, B \in\mathcal{M}_n$ be accretive matrices and $W(A), W(B)\subset S_{\alpha} $. Then, for $0<t<1$,
\begin{equation}
 \Re(A!_{t}B)\leq \sec^{2}(\alpha)(\Re A)!_{t}(\Re B).
\label{sec!}
\end{equation}
\end{lemma}
This lemma has been also extended to any matrix mean in Proposition \ref{prop_real_a_sigma_b_less}.
\begin{lemma}\label{lemma_sharp>} \cite{raissouli} Let $ A, B \in\mathcal{M}_n$ be accretive matrices and let $0<t<1.$ Then 
\begin{equation}
\Re(A\sharp_{t} B)\geq (\Re A)\sharp_{t}(\Re B).\label{sharp}
\end{equation}
\end{lemma}
\begin{lemma}\cite{Tan}\label{lemma_sharpsec} Let $ A, B \in\mathcal{M}_n$ be accretive matrices such that $W(A), W(B)\subset S_{\alpha} $. Then, for $0<t<1,$
\begin{equation}
 \Re(A\sharp_{t}B)\leq \sec^{2}(\alpha)(\Re A)\sharp_{t}(\Re B).
\label{sharpsec}
\end{equation}

\end{lemma}

It is well known that for any matrix $A\in\mathcal{M}_n$, $\|\Re A\|\leq \|A\|$, where $\|\cdot\|$ is any unitarily invariant norm  on $\mathcal{M}_n$. The following lemma presents a reversed version of this inequality for sectorial matrices.
\begin{lemma}\cite{Zhang}\label{norm}
Let $ A \in \mathcal{M}_{n} $ be such that $ W(A)\subset S_{\alpha}, $ for some $0\leq \alpha<\frac{\pi}{2}$ and let $ \parallel.\parallel $ be any unitarily invariant norm on $\mathcal{M}_n$. Then
\begin{center}
$ \cos(\alpha)\ \parallel A\parallel\ \leq\
 \parallel \Re(A)\parallel \leq \|A\|.$
\end{center}
\end{lemma}

\begin{lemma}\cite{Hoa} Let $ A, B\in \mathcal{M}_n $ be such that $ 0<m\mathcal{I}\leq A, B\leq M\mathcal{I} $, for some positive scalars $m,M$, and let $f,g\in\mathfrak{m}.$  Then for every unital positive linear map $ \Phi $,
\begin{equation}
\Phi^{2}(A\sigma_f B)\leq K(h)^{2}\Phi^{2}(A\sigma_g B),
\end{equation}
where $h=\frac{M}{m}$ and $k(h)=\frac{(h+1)^2}{4h}$ is the well known Kantorovich constant.
\end{lemma}
 Theorem \ref{theorem_square_sigma} below presents the  accretive version of this lemma.

\begin{lemma}\cite{Bhatia1}
Let $ A\in\mathcal{M}_n^+ $ and $ \Phi $ be positive linear map.  Then we have 
\begin{equation}
\Phi(A^{-1})\geq\ \Phi(A)^{-1}.\label{phi(A)}
\end{equation}
\end{lemma}
It is of potential interest to investigate the accretive version of this result. Theorem \ref{choi_davis_accretive} below present this interest for any matrix concave function.
\begin{lemma}\cite{Bhatia2}\label{normA+B}
Let $ A, B\in\mathcal{M}_n^+ $.  Then for any unitarily invariant norm $||\cdot||$, 
\begin{equation}
\parallel AB \parallel \leq\ \dfrac{1}{4}\parallel (A+B)^{2} \parallel. 
\end{equation}
\end{lemma}
The following characterization was given in \cite{ando_hiai} for matrix monotone functions.
\begin{lemma}\cite{ando_hiai}\label{lemma_ando_hiai} Let $ A, B \in\mathcal{M}_n^+ $  and $ f\in\mathfrak{m}. $  Then
\begin{equation}
f(A)\sharp f(B) \leq f(A\nabla B).
\end{equation}
\end{lemma}
The extension of this lemma to accretive matrices can be found in Theorem \ref{theorem_sharp_nabla}  below.

\begin{lemma}\label{lemma_f_of_norm_f}\cite{Hiai} Let $A\in\mathcal{M}_n^+$  and let $\|\cdot\|$ be a normalized unitarily invariant norm. If $f\in\mathfrak{m},$ then
$$f(\|A\|)\leq \|f(A)\|.$$
\end{lemma}
Proposition \ref{prop_f_norm_of_f}  below provides the accretive version of this lemma.

It is well-known that a concave function $f$ with $f(0)\geq 0$ is subadditive in the sense that 
\begin{equation}\label{conc_sub_intro}
f(a+b)\leq f(a)+f(b),
\end{equation} 
 for the non-negative numbers $a,b.$ A similar inequality is not necessarily valid for matrix concave functions. That is, an matrix concave function $f$ does not necessarily satisfy 
$$f(A+B)\leq f(A)+f(B),$$ for the positive matrices $A,B.$ In 1999, Ando and Zhan \cite{ando_zhan} proved a subadditivity inequality for $f\in\mathfrak{m}$.
\begin{lemma}\label{lemma_ando_zhan}\cite{ando_zhan} Let $A, B\in \mathcal{M}_n^+$. Then for any unitarily invariant norm ${{\left\| \cdot \right\|}}$ and any $f\in\mathfrak{m}$,
\begin{equation}\label{17}
{{\left\| f\left( A+B \right) \right\|}}\le {{\left\| f\left( A \right)+f\left( B \right) \right\|}}.
\end{equation}
\end{lemma}

 Bourin and Uchiyama \cite{bourin} showed that the condition matrix concavity in \eqref{17} can be replaced by scalar concavity. 

 The extension of \eqref{17} to accretive or sectorial matrices can be found in Theorem \ref{ando_zhan} below.
 
 In \cite{Gumus},  some inequalities among matrix means for positive matrices (i.e., $ ! $, $\sharp$, $\nabla$) were shown. We summarize these inequalities in the following proposition.
 \begin{proposition}\label{gumus}
 Let $m,M$ be positive scalars and let $A,B \in\mathcal{M}_n^+$ be such that $m\mathcal{I}\leq A,B\leq M\mathcal{I}.$ If $0<t<1$ and $\lambda=\min\lbrace t, 1-t\rbrace$, then
 \begin{align}\label{gumus_1}
A\nabla_{t} B\leq\dfrac{m\nabla_{\lambda}M}{m\sharp_{\lambda}M}A\sharp_{t}B,
\end{align}
\begin{align}\label{gumus_2}
A\sharp_{t} B\leq\dfrac{m\nabla_{\lambda}M}{m\sharp_{\lambda}M}A!_{t}B,
\end{align}
\begin{align}\label{gumus_3}
A\nabla_{t}B-M(\dfrac{m\nabla_{\lambda}M}{m\sharp_{\lambda}M}-1)\mathcal{I} \leq A\sharp_{t}B\leq M(\dfrac{m\nabla_{\lambda}M}{m\sharp_{\lambda}M}-1)\mathcal{I}+A!_{t}B.
\end{align}

 \end{proposition}
Propsoitions \ref{sbbh}, \ref{sbbh1} and \ref{sbbh2} discusses possible accretive versions of this last proposition.

Besides, in \cite{Ando},   some relations for positive definite matrices have been shown as follows:
 \begin{align}
 (A\nabla B)\sharp (A!B)=A\sharp B ,\label{sharpando}
 \end{align}
 \begin{align}
 A\nabla_{t}(A\sharp_{s}B)\geq A\sharp_{s}(A\nabla_{t}B).\label{ts}
 \end{align}
The accretive versions of these relations can be found in Theorems \ref{theorem_mixed_sharp_nabla_har} and \ref{theorem_mixed_nabla_sharp} below.

\section{An alternative formula for $f(A)$}
Inspired by the positive case, in this section we present an easier formula for $f(A)$, when $A$ is accretive and $f\in\mathfrak{m}.$ In fact, the result treats more general matrices than accretive ones, which is needed in our analysis. We recall that for such parameters, $f(A)$ is defined by \eqref{dunford_intro}. Our main result in this section reads as follows.
\begin{theorem}\label{thm_main_1}
Let $f\in\mathfrak{m}$ and $A\in\mathcal{M}_n$ be any matrix with eigenvalues set $\lambda(A).$ If $\lambda(A)\cap (-\infty,0]=\phi$, then
$$f(A)=\int_{0}^{1}\mathcal{I}!_{t}A\;d\nu_{f}(t),$$ where $\nu_f$ is a probability measure on $[0,1].$
\end{theorem}
\begin{proof}
First, we show the result when $A$ is a diagonalizable matrix with $\lambda(A)\cap (-\infty,0]=\phi.$ So, let $A$ be such matrix and let $f\in\mathfrak{m}.$ If $A=V^{-1}D[\lambda_i]V,$ where $D[\lambda_i]$ is diagonal, then immediate calculations show that
\begin{align*}
\int_{0}^{1}(\mathcal{I}!_tA)d\nu_f(t)&=V^{-1}D\left[\int_{0}^{1}(1!_s\lambda_i)d\nu_f(t)\right]V\\
&=V^{-1}D\left[\frac{1}{2\pi i}\int_{\Gamma}\frac{f(z)}{z-\lambda_i}dz\right]V\\
&=\frac{1}{2\pi i}\int_{\Gamma}f(z)(z\mathcal{I}-A)^{-1}dz,
\end{align*}
where we have used Proposition \ref{prop_our_hansen} to obtain the second identity, noting that $\lambda_i\not\in (-\infty,0]$. This shows the result for diagonalizable matrices with no eigenvalues in $(-\infty,0].$\\
For the general case, let $A\in\mathcal{M}_n$ with no eigenvalue in $(-\infty,0]$. Since diagonalizable matrices are dense in $\mathcal{M}_n$, with respect to the operator norm, we can find a sequence of diagonalizable matrices $(A_m)$ such that $A_m\to A$ in the usual operator norm, \cite[Corollary 5.1]{denis}. Further, since $A_m\to A$ and $\lambda(A)\cap (-\infty,0]=\phi$, we may assume without loss of generality that $\lambda(A_m)\cap (-\infty,0]=\phi$ for all $m$, due to the continuity of the map that maps each matrix  $A\in\mathcal{M}_n$ to its eigenvalues set $\lambda(A)$, \cite[Proposition 5.2.2]{martin}. Since $\lambda(A_m)\cap (-\infty,0]=\phi,$ and $A_m$ is diagonalizable, the first part of the proof implies that
\begin{align}\label{nee1}
f(A_m)=\int_{0}^{1}\mathcal{I}!_t A_m\;d\nu_f(t).
\end{align}
Since $f\in\mathfrak{m},$ it is analytically continued to $D:=\mathbb{C}\backslash (-\infty,0]$, and hence it is $n-1$ continuously differentiable in $D$. Therefore, the mapping $A\to f(A)$ is a continuous mapping on the set of matrices with spectrum in $D$, \cite[Theorem 1.19]{high}. This continuity implies that $f(A_m)\to f(A),$ which in turns implies (by \eqref{nee1})
\begin{equation}\label{nee2}
\int_{0}^{1}\mathcal{I}!_t A_m\;d\nu_f(t)\to \frac{1}{2\pi i}\int_{\Gamma}f(z)(z\mathcal{I}-A)^{-1}dz,
\end{equation}
where $\Gamma$ is a simple closed curve in $D$ that surrounds the spectrum of $A$. It remains to show that 
$$\int_{0}^{1}\mathcal{I}!_t A_m\;d\nu_f(t)\to \int_{0}^{1}\mathcal{I}!_t A\;d\nu_f(t).$$
We first notice that for each $t\in [0,1]$, the mapping $X\to \mathcal{I}!_tX$ is continuous on the class of matrices with spectrum in $D$, \cite[Theorem 1.19]{high}. Consequently, for every $t\in [0,1],$
\begin{equation}\label{needed_dunford_3}
\|\mathcal{I}!_tA_m-\mathcal{I}!_tA\|_\infty \to 0.
\end{equation}
This means that, for large $m$,
\begin{align*}
\|\mathcal{I}!_tA_m-\mathcal{I}!_tA\|_\infty&\leq \|\mathcal{I}!_tA_m\|_\infty+\|\mathcal{I}!_tA\|_\infty\\
&\leq 2 \|\mathcal{I}!_tA\|_\infty+1,
\end{align*}
where the above inequality follows from \eqref{needed_dunford_3}.
Noting the latter inequality and applying the dominated Lebesgue convergence theorem imply
$$\int_{0}^{1}\|\mathcal{I}!_tA_m-\mathcal{I}!_tA\|_\infty\;d\nu_f(t)\to 0,$$ which shows that 
\begin{align}\label{needed_dunford_4}
\int_{0}^{1}\mathcal{I}!_tA_md\nu_f(t)\to \int_{0}^{1}\mathcal{I}!_tAd\nu_f(t).
\end{align}
This together with \eqref{nee2} complete the proof of the theorem.
\end{proof}
Theorem \ref{thm_main_1} will be a key result in the subsequent sections. We notice that the statement of the theorem applies for accretive matrices, since the eigenvalues of such matrices are not in $(-\infty,0].$
\section{The geometric mean of accretive matrices}
In this section, we explore more properties of the geometric mean of accretive matrices. Our first observation is that the definition given in \eqref{raissouli_def} is consistent with \eqref{mean_har_prob_intro}. This provides a better understanding that geometric mean of accretive matrices follow the same rule as that of positive ones.

In \cite{raissouli}, it is shown that the definition of the weighted geometric mean for accretive matrices given in \eqref{raissouli_def} is equivalent to

\begin{align*}
    A\sharp_{\lambda}B=\int_{0}^{1}A!_{t}B\;d\nu_{\lambda}(t),
\end{align*}
where $d\nu_{\lambda}(t)=\frac{\sin(\lambda \pi)}{\pi}\frac{t^{\lambda-1}}{(1-t)^{\lambda}}dt$ is a probability measure.

The fact that $d\nu_{\lambda}(t)$ is a probability measure follows immediately on letting $A=B=\mathcal{I}.$

In the following result, we show that the definition in \eqref{raissouli_def} is consistent with that for positive matrices. It should be remarked that this result has been shown in \cite{drury} for $\lambda=\frac{1}{2}.$

\begin{theorem}\label{thm} Let $ A, B\in\mathcal{M}_n $ be two accretive matrices and $0<\lambda<1$. Then 
\begin{align}
A\sharp_{\lambda} B=A^{\frac{1}{2}}(A^{\frac{-1}{2}}BA^{\frac{-1}{2}})^{\lambda}A^{\frac{1}{2}},
\end{align}
where $(A^{\frac{-1}{2}}BA^{\frac{-1}{2}})^{\lambda}$ is defined via the Dunford integral as in \eqref{def_frac_power_intro}.
\end{theorem}
\begin{proof}
In order to use Theorem \ref{thm_main_1},  we first show that the eigenvalues of $A^{\frac{-1}{2}}BA^{\frac{-1}{2}}$ are not in $(-\infty,0].$ The proof of this fact was given in \cite{drury}, but we present it here for the sake of convenience for the reader. So, let $ \mu $ be an eigenvalue of $ A^{\frac{-1}{2}}BA^{\frac{-1}{2}} $. Then, there is a nonzero vector $ x $ such that $ A^{\frac{-1}{2}}BA^{\frac{-1}{2}}x=\mu x $. Putting $ y= A^{\frac{-1}{2}}x$, we get $ By=\mu Ay $. Therefore, $ y^{*}By=\mu y^{*}Ay  $, and since $ A,B $ are accretive matrices, then $ \mu $ does not lie on $(-\infty,0].$\\
Using $d\nu_{\lambda}(t)= \dfrac{\sin(\lambda\pi)}{\pi}\dfrac{t^{\lambda-1}}{(1-t)^{\lambda}},$ we have
\begin{align*}
A^{\frac{-1}{2}}(A\sharp_{\lambda} B)A^{\frac{-1}{2}} &=
\int_{0}^{1}A^{\frac{-1}{2}}(A!_{t}B)A^{\frac{-1}{2}}\ \ d\nu_{\lambda}(t)\nonumber\\
 &=\int_{0}^{1}((1-t)\mathcal{I}+tA^{\frac{1}{2}}B^{-1}A^{\frac{1}{2}})^{-1}\ \ d\nu_{\lambda}(t) \nonumber\\
 &=\int_{0}^{1}(\mathcal{I}!_{t}A^{\frac{-1}{2}}BA^{\frac{-1}{2}})\ \ d\nu_{\lambda}(t)\nonumber\\
 &=(A^{\frac{-1}{2}}BA^{\frac{-1}{2}})^{\lambda}\hspace{5cm} (\text{by\; Theorem \ref{thm_main_1}}).\nonumber
\end{align*}
This implies
\[ A\sharp_{\lambda} B=A^{\frac{1}{2}}(A^{\frac{-1}{2}}BA^{\frac{-1}{2}})^{\lambda}A^{\frac{1}{2}},\] as desired.
\end{proof}

Having shown Theorem \ref{thm}, we can deal with the definition of the geometric mean in a similar manner, whether our matrices are positive or accretive. This allows us to obtain many other properties for the geometric mean for accretive matrices, which are similar to those for positive ones. For example, the following applies.

\begin{corollary} Let $ A,B\in\mathcal{M}_n $ be accretive matrices. Then for $\lambda\in (0,1),$ 
\[(A\sharp_{\lambda} B)^{-1}=A^{-1}\sharp_{\lambda} B^{-1}.\]
\end{corollary}
\begin{proof}
 This is an immediate consequence of Theorem \ref{thm}.
\end{proof}

Referring to the literature dealing with geometric mean of positive matrices, we find a considerable attention to $\sharp_{\lambda}$ when $\lambda\not\in [0,1].$ In the next definition, we present the corresponding definition for accretive matrices.
\begin{definition} Let $ A,B\in\mathcal{M}_n$ be accretive matrices and let $\lambda\in\mathbb{R}$. We define
\begin{align*}
    A\sharp_{\lambda}B=A^{\frac{1}{2}} \left( A^{\frac{-1}{2}}BA^{\frac{-1}{2}}\right)^{\lambda}A^{\frac{1}{2}}.
\end{align*}
\end{definition}
In particular, we have:
\begin{proposition} If $ A,B\in\mathcal{M}_n$ are accretive matrices and $0<\lambda<1$, then 
\begin{align}
A\sharp_{-\lambda}B=A\left\{\dfrac{\sin(\lambda\pi)}{\pi}\int_{0}^{1}\dfrac{t^{\lambda-1}}{(1-t)^{\lambda}}A^{-1}!_{t}B^{-1}\ \ dt\right\}A.\label{-t}
\end{align}
In particular, \eqref{-t} holds when $A,B\in\mathcal{M}_n^+$. 
\end{proposition} 
\begin{proof}
 For $ 0<\lambda <1$, we have
\begin{align*}
\displaystyle
A\sharp_{-\lambda}B &=A^{\frac{1}{2}} \left( A^{\frac{-1}{2}}BA^{\frac{-1}{2}}\right)^{-\lambda} A^{\frac{1}{2}} \\
&=A^{\frac{1}{2}}\left[ \dfrac{\sin(\lambda \pi)}{\pi}\int_{0}^{1}\dfrac{t^{\lambda-1}}{(1-t)^{\lambda}}\ \ \mathcal{I}!_{t}\left( A^{\frac{-1}{2}}BA^{\frac{-1}{2}}\right)^{-1} \ \ dt\right]A^{\frac{1}{2}}\\
&=A^{\frac{1}{2}}\left[ \dfrac{\sin(\lambda \pi)}{\pi}\int_{0}^{1}\dfrac{t^{\lambda-1}}{(1-t)^{\lambda}}\ \ \mathcal{I}!_{t}(A^{\frac{1}{2}}B^{-1}A^{\frac{1}{2}})\ \ dt\right]A^{\frac{1}{2}}\\
&=A\left[ \dfrac{\sin(\lambda \pi)}{\pi}\int_{0}^{1}\dfrac{t^{\lambda-1}}{(1-t)^{\lambda}}A^{-1}!_{t}B^{-1}\ \ dt\right] A.
\end{align*}
\end{proof} 

For the rest of this section, we will present several inequalities for the geometric mean of accretive matrices. These inequalities simulate similar results for positive ones. 

The following three propositions present the accretive version of Proposition \ref{gumus}.
\begin{proposition}\label{sbbh} Let $ A, B \in\mathcal{M}_{n}$ be accretive matrices such that $ 0<m\mathcal{I}\leq \Re A, \Re B\leq M\mathcal{I} $, for some scalars $ 0<m<M $. If $0<t<1$ and $ \lambda=\min\lbrace t, 1-t\rbrace$, then
\begin{equation}
\Re(A\nabla_{t} B)\leq\dfrac{m\nabla_{\lambda}M}{m\sharp_{\lambda}M}\Re(A\sharp_{t}B).
\end{equation}
\end{proposition}
\begin{proof} 
Let $ \lambda=\min\lbrace t, 1-t\rbrace$ for $t\in(0, 1). $ Then
\begin{eqnarray}
\Re(A\nabla_{t} B)&=&\Re A\nabla_{t}\Re B\leq\dfrac{m\nabla_{\lambda}M}{m\sharp_{\lambda}M}(\Re A\sharp_{t}\Re B)\leq \dfrac{m\nabla_{\lambda}M}{m\sharp_{\lambda}M}\Re(A\sharp_{t}B),\nonumber
\end{eqnarray}
where we have used \eqref{gumus_1} and Lemma \ref{lemma_sharp>} to obtain the first and second inequalities, respectively.
\end{proof}
\begin{proposition}\label{sbbh1} Let $ A, B \in\mathcal{M}_{n}$ be accretive matrices such that $ 0<m\mathcal{I}\leq \Re A, \Re B\leq M\mathcal{I} $, for some scalars $ 0<m<M $, and $ W(A), W(B)\subset S_{\alpha},$ for some $0\leq\alpha<\frac{\pi}{2}$. If $0<t<1$ and $ \lambda=\min\lbrace t, 1-t\rbrace$, then
\begin{equation}
\Re(A\sharp_{t} B)\leq\sec^{2}(\alpha)\dfrac{m\nabla_{\lambda}M}{m\sharp_{\lambda}M}\Re(A!_{t}B).
\end{equation}

\end{proposition}
\begin{proof} 
Let $ \lambda=\min\lbrace t, 1-t\rbrace$ for $0<t<1.$ Then
\begin{eqnarray}
\Re(A\sharp_{t}B)&\leq &\sec^{2}(\alpha)(\Re A\sharp_{t}\Re B)
\leq \sec^{2}(\alpha)\dfrac{m\nabla_{\lambda}M}{m\sharp_{\lambda}M}(\Re A!_{t}\Re B)
\leq \sec^{2}(\alpha)\dfrac{m\nabla_{\lambda}M}{m\sharp_{\lambda}M}\Re(A!_{t}B),\nonumber
\end{eqnarray}
where we have used Lemma \ref{lemma_sharpsec}, \eqref{gumus_2} and Lemma \ref{2} to obtain the first, second and third inequalities, respectively.
\end{proof}
\begin{proposition}\label{sbbh2}  Let $ A, B \in\mathcal{M}_{n}$ be accretive matrices such that $ 0<m\mathcal{I}\leq \Re A, \Re B\leq M\mathcal{I} $, for some scalars $ 0<m<M $, and $ W(A), W(B)\subset S_{\alpha},$ for some $0\leq \alpha<\frac{\pi}{2}$.  If $0<t<1$ and  $ \lambda=\min\lbrace t, 1-t\rbrace$, then
\begin{equation}
\Re(A\nabla_{t}B)-M\left(\dfrac{m\nabla_{\lambda}M}{m\sharp_{\lambda}M}-1\right)\mathcal{I} \leq \Re(A\sharp_{t}B)\leq \sec^{2}(\alpha)\left( M\left(\dfrac{m\nabla_{\lambda}M}{m\sharp_{\lambda}M}-1\right)\mathcal{I}+\Re(A!_{t}B)\right).
\end{equation}
\end{proposition}
\begin{proof} 
 Following the same reasoning as in the proof of the above propositions, we have
\begin{eqnarray}
\Re(A\nabla_{t}B)-M\left(\dfrac{m\nabla_{\lambda}M}{m\sharp_{\lambda}M}-1\right)\mathcal{I} &=&\Re A\nabla_{t}\Re B-M\left(\dfrac{m\nabla_{\lambda}M}{m\sharp_{\lambda}M}-1\right)\mathcal{I}\nonumber\\
&\leq &\Re A\sharp_{t}\Re B\leq \Re (A\sharp_{t}B).\nonumber
\end{eqnarray}
This proves the first inequality. For the second inequality, notice that
\begin{eqnarray}
M\left(\dfrac{m\nabla_{\lambda}M}{m\sharp_{\lambda}M}-1\right)\mathcal{I}+\Re(A!_{t}B)&\geq & M\left(\dfrac{m\nabla_{\lambda}M}{m\sharp_{\lambda}M}-1\right)\mathcal{I}+\Re A!_{t}\Re B\nonumber\\
&\geq &\Re A\sharp_{t}\Re B\geq  \cos^{2}(\alpha)\Re (A\sharp_{t}B),\nonumber
\end{eqnarray}
This completes the proof.
\end{proof}
The following two theorems present the accretive versions of relations \eqref{sharpando} and \eqref{ts} respectively.
\begin{theorem}\label{theorem_mixed_sharp_nabla_har} Let $ A, B\in\mathcal{M}_n $ be accretive matrices such that  $ W(A), W(B)\subset S_{\alpha},$ for some $0\leq \alpha<\frac{\pi}{2}$. Then 
\begin{equation}
\cos^{3}(\alpha)\Re[(A\nabla B)\sharp(A!B)]\leq \Re(A\sharp B)\leq \sec^{2}(\alpha)\Re [(A\nabla B)\sharp (A!B)]. 
\end{equation}
\end{theorem}
\begin{proof}
First, 
\begin{align*}
\Re(A\sharp B)&\geq\Re A\sharp\Re B\hspace{3.5cm}({\text{by Lemma}}\;\ref{lemma_sharp>})\\
&=(\Re A\nabla\Re B)\sharp(\Re A!\Re B)\hspace{2cm}({\text{by}}\;\eqref{sharpando})\\
&\geq\Re(A\nabla B)\sharp(\cos^{2}(\alpha)\Re(A!B))\hspace{0.4cm}({\text{by Lemma}}\;\ref{lemma_reverse_real_!t})\\
&=\cos(\alpha)\Re(A\nabla B)\sharp\Re(A!B)\\
&\geq \cos^{3}(\alpha)\Re[(A\nabla B)\sharp(A!B)]
\end{align*}
where we have used Lemma \ref{lemma_sharpsec} to obtain the last inequality. Notice that the use of Lemma \ref{lemma_sharpsec} here is justified since $S_{\alpha}$ is invariant under addition and inversion, which ensures that when $W(A),W(B)\subset S_{\alpha}$ then $W(A!B)\subset S_{\alpha};$ see \cite[Proposition 3.2]{sab-lama}. This proves the first inequality. For the second inequality,
\begin{align*}
\Re \left((A\nabla B)\sharp (A!B)\right)&\geq\Re(A\nabla B)\sharp \Re(A!B) \hspace{1.5cm}({\text{by Lemma}}\;\ref{lemma_sharp>})\\
&\geq\Re(A\nabla B)\sharp (\Re A !\Re B) \\
&=(\Re A\nabla\Re B)\sharp (\Re A !\Re B)\\
&= \Re A \sharp \Re B \  \hspace{3cm} (\text{by}\; \eqref{sharpando})\\
&\geq \cos^{2}(\alpha)\ \Re(A\sharp B),\hspace{2cm} ({\text{by Lemma}}\;\ref{lemma_sharpsec})
\end{align*}
which completes the proof.
\end{proof}

\begin{theorem}\label{theorem_mixed_nabla_sharp}  Let $ A, B \in\mathcal{M}_n $ be accretive matrices such that $ W(A), W(B)\subset S_{\alpha},$ for some $0\leq \alpha<\frac{\pi}{2}$. Then for $ t,s\in(0,1) $,
\begin{equation}
\Re(A\sharp_{s}(A\nabla_{t}B))\leq \sec^{2}(\alpha)\Re(A\nabla_{t}(A\sharp_{s}B)).
\end{equation}
\end{theorem}
\begin{proof} We have,
\begin{eqnarray}
\Re(A\nabla_{t}(A\sharp_{s}B))&=&\Re A\nabla_{t}\Re(A\sharp_{s}B)\nonumber\\
&\geq &\Re A\nabla_{t}(\Re A\sharp_{s}\Re B)\nonumber\\
&\geq &\Re A\sharp_{s}(\Re A\nabla_{t}\Re B)\hspace{2cm} (\text{by}\; \eqref{ts})\nonumber\\
&=&\Re A\sharp_{s}\Re(A\nabla_{t}B)\nonumber\\
&\geq & \cos^{2}(\alpha)\ \Re(A\sharp_{s}(A\nabla_{t}B)),\hspace{2cm} (\text{by}\; \eqref{sharpsec})\nonumber
\end{eqnarray}
which completes the proof.
\end{proof}  

\section{Arbitrary means of accretive matrices}
We have introduced matrix means for positive matrices earlier in the introduction, and we have seen that if $f\in \mathfrak{m}$, then a probability measure $\nu_f$ exists such that for positive  $A,B$, one has
\begin{align}\label{nee3}
    A\sigma_f B&=A^{\frac{1}{2}}f\left(A^{-\frac{1}{2}}BA^{-\frac{1}{2}}\right)A^{\frac{1}{2}}\\
    &=\int_{0}^{1}A!_{t}B\;d\nu_f(t).
\end{align}
Also, we have discussed the geometric mean of accretive matrices following this point of view.

Our goal in this section is to extend the definition of  an arbitrary matrix mean to the context of accretive matrices. This study will generalize the geometric mean idea to all matrix means. Our central definition in this section reads as follows.
\begin{definition}\label{def_sigma_f}
Let $A,B\in\mathcal{M}_n$ be two accretive matrices, $f\in \mathfrak{m}$, and let $\nu_f$ be the probability measure as in Lemma \ref{hansen_lem}. We define the matrix mean $\sigma_f$, induced by $f$, of $A$ and $B$ by
$$A\sigma_f B=\int_{0}^{1}A!_tB\;d\nu_f(t).$$
\end{definition}

\begin{remark}
Our first remark is that we adopt the above defintion for accretive matrices only. Notice that for $A\sigma_f B$ to be defined, we must have $A!_t B$ defined for all $t\in [0,1].$ This means that we must have $(1-t)A^{-1}+tB^{-1}$ invertible, for all $t\in [0,1].$ When $A$ and $B$ are both accretive, this is guaranteed. However, if they are not accretive, we have no control over this. This is the main reason we restrict ourselves to accretive matrices in this definition, and in the following discussion.
\end{remark}

When $A$ and $B$ are accretive, Drury \cite{drury} showed that the spectrum of the matrix $A^{-1/2}BA^{-1/2}$ is disjoint from $(-\infty,0]$ (see Theorem \ref{thm}); justifying the use of $f\left(A^{-1/2}BA^{-1/2}\right)$ in the following result, with the aid of Theorem \ref{thm_main_1}. In what follows, we show the accretive version of \eqref{nee3}, as a main application of Theorem \ref{thm_main_1}.
 \begin{theorem}
 Let $A,B\in\mathcal{M}_n$ be accretive matrices and let $f\in \mathfrak{m}.$ Then
 $$A\sigma_f B=A^{\frac{1}{2}}f\left(A^{-\frac{1}{2}}BA^{-\frac{1}{2}}\right)A^{\frac{1}{2}}.$$
 \end{theorem}
\begin{proof}
By Definition \ref{def_sigma_f}, we have
\begin{align*}
A\sigma_f B&=\int_{0}^{1}A!_t B\;d\nu_f(t)\\
&=\int_{0}^{1}\left((1-t)A^{-1}+tB^{-1}\right)^{-1}d\nu_f(t)\\
&=A^{\frac{1}{2}}\int_{0}^{1}\left((1-t)\mathcal{I}+t\left(A^{-\frac{1}{2}}BA^{-\frac{1}{2}}\right)^{-1}\right)^{-1}d\nu_f(t)A^{\frac{1}{2}}\\
&=A^{\frac{1}{2}}f\left(A^{-\frac{1}{2}}BA^{-\frac{1}{2}}\right)A^{\frac{1}{2}}, 
\end{align*}
where we have used Theorem \ref{thm_main_1} to obtain the last identity. This completes the proof.
\end{proof}
We refer the reader to  \cite{sano}, where operator means of accretive operators on Hilbert spaces were treated. Our definition coincides with the main definition in \cite{sano} when we restrict ourselves to $\mathcal{M}_n$.

Now we begin our investigation by reciting the following result which extends Lemma \ref{2} to any matrix mean.
\begin{proposition}\label{r_a_sigma_b>} Let $ A, B \in \mathcal{M}_n $ be accretive matrices and let $f\in \mathfrak{m}.$ Then  
\begin{equation}
\Re(A\sigma_f B)\geq (\Re A)\;\sigma_f\;(\Re B).\label{3}
\end{equation}
As a consequence, if $A$ and $B$ are accretive, then so is $A\sigma_f B$.
\end{proposition}
\begin{proof}
Let $ A, B \in \mathcal{M}_n $ be accretive. Then 
\begin{eqnarray}
\displaystyle
\Re(A\sigma_f B)&=&\int_{0}^{1}\Re( A!_{t}B) \; d\nu_f(t) \nonumber \\
&\geq &\int_{0}^{1}\Re( A)!_{t}\Re(B) \; d\nu_f(t) \hspace{1.5cm}(\text{by Lemma}\;\ref{2})\nonumber \\
&=&(\Re A)\;\sigma_f\;(\Re B).\nonumber
\end{eqnarray}
This completes the proof.
\end{proof}

When $A$ and $B$ are sectorial, we have the following reverse of Proposition \ref{r_a_sigma_b>}. 
\begin{proposition} \label{prop_real_a_sigma_b_less}
Let $ A, B\in\mathcal{M}_n $ be accretive matrices such that  $ W(A), W(B)\subset S_{\alpha},$ for some $0\leq \alpha<\frac{\pi}{2}$. If $f\in \mathfrak{m}$, then
\begin{equation}
\Re(A\sigma_f B)\leq \sec^{2}(\alpha)\;(\Re A)\;\sigma_f\;(\Re B).\label{rleqsec}
\end{equation}
\end{proposition}
\begin{proof}
By Definition \ref{def_sigma_f}, we have 
\begin{eqnarray}
\displaystyle
\Re(A\sigma_f B)&=&\int_{0}^{1}\Re( A!_{t}B) \; d\nu_f(t) \nonumber \\
&\leq &\sec^{2}(\alpha)\int_{0}^{1} \;\left( \Re(A)!_{t}\Re(B)\right)  \; d\nu_f(t)\hspace{0.4cm}({\text{by Lemma}}\;\ref{lemma_reverse_real_!t})\nonumber \\
&=&\sec^{2}(\alpha)\;(\Re A)\;\sigma_f\; (\Re B)\nonumber .
\end{eqnarray}
This completes the proof.
\end{proof}

Now, we present a generalization of Lemma \ref{AM-GM-HM} from the setting of positive matrices to  sectorial ones.

\begin{theorem}\label{thm_amgmhm_accretive}
Let $ A, B\in\mathcal{M}_n $ be accretive matrices such that $W(A), W(B)\subset S_{\alpha}$ for some $0\leq \alpha<\frac{\pi}{2}$. If $f\in \mathfrak{m}$ is such that  $f'(1)=t$ for some $t\in (0,1),$ then 
\begin{equation}
 \cos^{2}(\alpha)\;\Re(A!_{t}B)   \ \leq\ \Re(A\sigma_f B)\ \leq\ \sec^{2}(\alpha)\;\Re(A\nabla_{t}B).\label{amgmhm_accretive}   
\end{equation}
\end{theorem}
\begin{proof} First,
\begin{eqnarray}
\displaystyle
\Re(A!_{t}B)&\leq &\sec^{2}(\alpha) \left((\Re A)!_{t}(\Re B)\right) \hspace{4cm}(\text{by Lemma}\;\ref{lemma_reverse_real_!t}) \nonumber\\
&\leq &\sec^{2}(\alpha) \left((\Re A)\sigma_f(\Re B)\right) \hspace{3cm}(\text{by Lemma}\;\ref{AM-GM-HM})\nonumber\\
&\leq &\sec^{2}(\alpha)\;\Re (A\sigma_f B).\hspace{4cm}(\text{by}\;\eqref{3})\nonumber
\end{eqnarray}
Thus, we have shown the first inequality. To show the second inequality, we have
\begin{eqnarray}
\displaystyle
\Re(A\sigma_f B)&\leq &\sec^{2}(\alpha) (\Re A)\;\sigma_f\;(\Re B)\hspace{2cm}(\text{by}\;\eqref{rleqsec})\nonumber \\
&\leq &\sec^{2}(\alpha)\; (\Re A)\nabla_{t}(\Re B)
\hspace{2cm}(\text{by Lemma}\; \ref{AM-GM-HM})\nonumber\\
&= &\sec^{2}(\alpha)\;\Re(A\nabla_{t}B).\nonumber
\end{eqnarray}
This shows the second desired inequality, and the proof is complete.
\end{proof}
We notice that when $A,B$ are positive, then $\alpha$ can be taken as $\alpha=0$, which then retrieves Lemma \ref{AM-GM-HM} as a special case of Theorem \ref{thm_amgmhm_accretive}.

The next result is a monotonic result for matrix means of accretive matrices. This result simulates the same known conclusion for positive ones.
  \begin{theorem} Let $ A, B, C, D\in\mathcal{M}_n $ be accretive matrices such that $ \Re A\leq \Re C, \Re B\leq \Re D $ and $\;W(A), W(B)\subset S_{\alpha}$ for some $0\leq \alpha<\frac{\pi}{2}$. If $f\in \mathfrak{m}$, then 
 \begin{equation}
 \Re(A\sigma_f B)\leq\ \sec^{2}(\alpha)\; \Re(C\sigma_f D).
 \end{equation}
\end{theorem}
\begin{proof} We have
\begin{eqnarray}
\displaystyle
\Re(A\sigma_f B)&=&\Re\left(\int_{0}^{1} A!_{t}B \; d\nu_f(t) \right)\nonumber \\
&\leq &\sec^{2}(\alpha) \int_{0}^{1}\Re( A)!_{t}\Re(B) \; d\nu_f(t)\hspace{3cm}(\text{by Lemma}\;\ref{lemma_reverse_real_!t})\nonumber \\
&\leq &\sec^{2}(\alpha) \int_{0}^{1}\Re(C)!_{t}\Re(D) \; d\nu_f(t) \nonumber \\
&=&\sec^{2}(\alpha) (\Re C)\sigma_f (\Re D)\nonumber \\
&\leq &\sec^{2}(\alpha)\;\Re( C\sigma_f D),\nonumber 
\end{eqnarray}
where in the above proof, we have used the fact that $\sigma_f$ is monotone on $\mathcal{M}_n^+$, justifying the inequality $\Re(A)!_t\Re(B)\leq \Re(C)!_t\Re(D).$ This completes the proof.
\end{proof}
Next we show the so called ``transformer identity" for matrix means of accretive matrices. This result, again, simulates the corresponding result for positive matrices. We first make the following observation. If $A$ is accretive and $C$ is any matrix, we have for any vector $x$,
$\left<C^*AC\;x,x\right>=\left<A(Cx),Cx\right>,$ which belongs to the right-half complex plane, since $A$ is accretive. This shows that when $A$ is accretive and $C$ is any matrix, then $C^*AC$ is also accretive. We can then state the following result.
\begin{theorem} Let $ A, B\in\mathcal{M}_n $ be accretive and let $f\in \mathfrak{m}$. Then for any invertible   $ C\in\mathcal{M}_n $,
\begin{equation}
C^{*}(A\sigma_f B)C=(C^{*}AC)\sigma_f (C^{*}BC).
\end{equation}
\end{theorem}
\begin{proof}
Let $ C\in\mathcal{M}_n $ be invertible. Then
\begin{eqnarray}
\displaystyle
C^{*}(A\sigma_f B)C&=&C^{*}\left(\int_{0}^{1} A!_{t}B \; d\nu_f(t)\right)C \nonumber \\
&= &\int_{0}^{1}C^{*}(A !_{t} B)C \; d\nu_f(t)  \nonumber \\
&= &\int_{0}^{1}(C^{*}AC)!_{t}(C^{*}BC) \; d\nu_f(t) \nonumber \\
&=&(C^{*}AC)\sigma_f (C^{*}BC),\nonumber 
\end{eqnarray}
which completes the proof.
\end{proof}

In studying matrix means, it is customary to compare between different means that arise from different matrix monotone functions. In the next result, we present such comparison for sectorial matrices.

\begin{theorem}\label{theorem_square_sigma}
 Let $ A, B\in\mathcal{M}_n $ be accretive matrices such that $0<m\mathcal{I}\leq \Re A, \Re B\leq M\mathcal{I} $ and $ W(A), W(B)\subset S_{\alpha} $ for some $0\leq \alpha<\frac{\pi}{2}$. If $f,g\in \mathfrak{m}$. Then for every unital positive linear map $ \Phi $,
\begin{equation}
\left\| \Phi(\Re(A\sigma_{f} B))\Phi(\Re(A\sigma_{g} B))^{-1}\right\|_\infty \leq \sec^{6}(\alpha)K(m,M) ,\label{tau}
\end{equation}
where $\|\cdot\|_{\infty}$ is the usual operator norm and $K(m,M)=\dfrac{(M+m)^{2}}{4Mm}. $
\end{theorem}
\begin{proof} Since $0<m\mathcal{I}\leq\ \Re A\leq M\mathcal{I},$\\
we have $$(M-\Re A)(m-\Re A)(\Re A)^{-1}\leq 0 ,$$
which is equivalent to
 $$ \Re A+Mm(\Re A)^{-1}\leq (M+m)\mathcal{I},$$
since $(\Re A)^{-1}\geq \Re (A^{-1})$ by Lemma \ref{RA}, we have 
\begin{align}
\dfrac{1}{2}\Re A+\dfrac{1}{2}Mm\Re (A^{-1})\leq \dfrac{1}{2}(M+m)\mathcal{I}.\label{*}
\end{align}
Similarly
\begin{align}
\dfrac{1}{2}\Re B+\dfrac{1}{2}Mm\Re (B^{-1})\leq \dfrac{1}{2}(M+m)\mathcal{I}.\label{**}
\end{align}
Adding \eqref{*} and \eqref{**}, we get  
\begin{equation}
\Re(A\nabla B)+Mm\Re(A!B)^{-1}\leq\ (M+m)\mathcal{I}. \label{M+m}
\end{equation}

Letting $\|\cdot\|_{\infty}$ denote the usual operator norm, we have

\begin{eqnarray}
&&\parallel \sec^{2}(\alpha)Mm\Phi(\Re(A\sigma_{f} B))\Phi(\Re(A\sigma_{g} B))^{-1}\parallel_\infty \nonumber\\
\displaystyle
&\leq &\dfrac{1}{4}\parallel \sec^{2}(\alpha)\Phi(\Re(A\sigma_{f} B))+Mm\Phi(\Re(A\sigma_{g} B))^{-1}\parallel_{\infty}^{2}\hspace{2cm}(\text{by Lemma}\; \ref{normA+B})\nonumber\\
&\leq &\dfrac{1}{4}\parallel \sec^{2}(\alpha)\Phi(\Re(A\sigma_{f} B))+Mm\Phi((\Re(A\sigma_{g} B))^{-1})\parallel_{\infty}^{2}\hspace{2cm}(\text{by \eqref{phi(A)}})\nonumber\\
&\leq &\dfrac{1}{4}\parallel \sec^{4}(\alpha)\Phi(\Re(A\nabla B))+\sec^{2}(\alpha)Mm\Phi(\Re(A! B))^{-1}\parallel_{\infty}^{2}\hspace{1cm}(\text{by}\; \eqref{amgmhm_accretive})\nonumber\\
&\leq &\dfrac{1}{4}\parallel \sec^{4}(\alpha)\Phi(\Re(A\nabla B))+\sec^{4}(\alpha)Mm\Phi(\Re(A! B)^{-1})\parallel_{\infty}^{2}\nonumber\\
&\leq &\dfrac{1}{4}\parallel \sec^{4}(\alpha)\Phi(\Re(A\nabla B)+Mm\Re(A! B)^{-1})\parallel_{\infty}^{2}\nonumber\\
&\leq &\dfrac{1}{4}\sec^{8}(\alpha)(M+m)^{2}\nonumber\hspace{7cm}(\text{by} \eqref{M+m}).\nonumber
\end{eqnarray}
That is
$$\parallel \Phi(\Re(A\sigma_{f} B))\Phi(\Re(A\sigma_{g} B))^{-1}\parallel_\infty\leq \sec^{6}(\alpha)K(m,M), $$
which completes the proof.
 
\end{proof}

\section{Ando-type inequalities for accretive matrices}
In this section we present versions of Ando's inequality \eqref{andos_inequality_intro}.
We begin by stating the following needed lemma which concerns the Ando-type
inequality for the harmonic matrix mean. For this purpose, we notice that if $\Phi$ is a positive linear map and $A$ is any matrix, then 
\begin{equation}\label{real_phi_real}
\Re \Phi(A)=\Phi(\Re A).
\end{equation}
\begin{lemma}  Let $ A, B \in \mathcal{M}_{n} $ be accretive and let $\Phi$ be a unital positive linear map. Then
\begin{equation}
\Phi(\Re A!_{t}\Re B)\leq \Re(\Phi(A)!_{t}\Phi(B)).\label{phi!!}
\end{equation}
\end{lemma}
\begin{proof}
Noting \eqref{andos_inequality_intro}, \eqref{real_phi_real} and Lemma \ref{2}, we have
\begin{eqnarray}
\displaystyle
\Phi(\Re A!_{t}\Re B)\leq \Phi (\Re A)!_{t}\Phi(\Re B) =\Re \Phi(A)!_{t}\Re \Phi(B) \leq \Re(\Phi(A)!_{t}\Phi(B)) .
\end{eqnarray}
\end{proof}
Now we are in the position
to state the sectorial  version of \eqref{andos_inequality_intro}, valid for any matrix mean.
\begin{theorem}\label{thm_Phicos}
Let $ A, B \in \mathcal{M}_{n} $ be accretive matrices such that $ W(A), W(B)\subset S_{\alpha},\; 0\leq \alpha<\frac{\pi}{2} $ and let $\Phi$ be a positive linear map. If $f\in\mathfrak{m},$ then
\begin{equation}
  \Re \Phi(A\sigma_f B)\leq \sec^{2}(\alpha)\; \Re
\left( \Phi(A)\sigma_f \Phi(B)\right).\label{Phicos}
\end{equation}
\end{theorem}
\begin{proof}
We have
\begin{eqnarray*}
\displaystyle
\cos^{2}(\alpha)\;\Re \Phi(A\sigma_f B)&=&\Phi(\cos^2(\alpha)\;\Re(A\sigma_fB))\hspace{0.4cm}({\text{by}}\;\eqref{real_phi_real})\\
&\leq& \Phi(\Re A\sigma_f\Re B) \hspace{0.9cm}({\text{by Proposition}}\;\ref{prop_real_a_sigma_b_less})\\
&\leq&\Phi(\Re A)\sigma_f\Phi(\Re B) \hspace{0.4cm}({\text{by}}\;\eqref{andos_inequality_intro})\\
&=&\Re \Phi(A)\sigma_f\Re \Phi(B) \hspace{0.4cm}({\text{by}}\;\eqref{real_phi_real})\\
&\leq& \Re(\Phi(A)\sigma_f \Phi(B)) \hspace{0.4cm}({\text{by Proposition}}\;\ref{r_a_sigma_b>}),
\end{eqnarray*}
which completes the proof.
\end{proof}
As an application of Theorem \ref{thm_Phicos}, we present the following accretive version of Lemma \ref{<sigma_f>}.
 \begin{corollary}\label{theorem_sigma_inner}
Let $ A, B\in\mathcal{M}_n  $ be accretive matrices such that $W(A), W(B)\subset S_{\alpha}$ for some $0\leq \alpha<\frac{\pi}{2}$. If $f\in \mathfrak{m}$, then for any vector $x,$ we have 
\begin{equation}
\Re\left\langle (A\sigma_f B)x,x\right\rangle \leq \sec^2(\alpha)\Re\left(\left\langle Ax,x\right\rangle \sigma_f\left\langle Bx,x\right\rangle\right) .
\end{equation}
\end{corollary}
\begin{proof} 
 Let  $\Phi(A)=\left\langle Ax,x\right\rangle$ in Theorem \ref{thm_Phicos}. Then $\Phi$ is a  positive linear map and 
$$\cos^{2}(\alpha)\  \Re \Phi(A\sigma_f B)\leq \Re
\left( \Phi(A)\sigma_f \Phi(B)\right)\Rightarrow \cos^2(\alpha)\Re\left\langle (A\sigma_f B)x,x\right\rangle \leq \Re\left(\left\langle Ax,x\right\rangle \sigma_f\left\langle Bx,x\right\rangle\right),$$ 
which completes the proof
\end{proof}

\begin{theorem} Let $ A, B\in\mathcal{M}_n $ be accretive matrices such that $W(A), W(B)\subset S_{\alpha}$ for some $0\leq \alpha<\frac{\pi}{2}$. If $f\in \mathfrak{m}$ is such that  $f'(1)=t$ for some $t\in (0,1).$ Then, for any positive linear map $\Phi$,
\begin{equation}
\Re \Phi(A\sigma_f B)\leq\sec^{2}(\alpha)\ \Re\Phi(A\nabla_t B).\label{sigmanabla}
\end{equation}
\end{theorem}
\begin{proof} We have
\begin{align*}
\cos^{2}(\alpha)\Phi(\Re(A\sigma_f B))&\leq\Phi(\Re A\sigma_f\Re B)\hspace{1cm}(\text{by Proposition}\; \ref{prop_real_a_sigma_b_less})\\ 
&\leq \Phi(\Re A)\sigma_f\Phi(\Re B) \\
&\leq \Phi(\Re A)\nabla_t\Phi(\Re B)\hspace{1cm}(\text{by Lemma}\;  \ref{AM-GM-HM})\\
&=\Re \left(\Phi(A)\nabla_t\Phi(B)\right)\\
&=\Re \Phi(A\nabla_t B),
\end{align*}
which completes the proof.
\end{proof}

\section{Choi-Davis inequalities for accretive matrices} 
In this section we present several inequalities involving $f(A)$, where $f\in\mathfrak{m}$ and $A$ is accretive. 

 For completeness of the proof, it is important to recall that a function $f\in \mathfrak{m}$ can be analytically continued to $\mathbb{C}\backslash (-\infty,0]$. This means that $f(A)$ can be defined similarly for any  $A$ whose spectrum is disjoint from $(-\infty,0]$, by Theorem \ref{thm_main_1}.

Our first result in this direction is the following relation between $f(\Re A)$ and $\Re(f(A)).$
\begin{proposition}\label{prop_f_r_f} Let $ f\in \mathfrak{m} $ and $ A\in\mathcal{M}_n $ be accretive. Then  
\begin{equation}
\Re (f(A))\geq\ f(\Re A).
\end{equation}
Consequently, if $A$ is accretive, then so is $f(A)$.
\end{proposition}
\begin{proof}
We easily notice that
\begin{equation}
\displaystyle
\Re f(A)=\int_{0}^{1}\Re({\mathcal{I}}!_{t} A)\;
d\nu_f(t)\geq \int_{0}^{1}\mathcal{I}!_{t}(\Re A) \;d\nu_f(t)=f(\Re A),
\end{equation}
where we have used Lemma \ref{2} to obtain the first inequality.
\end{proof}
On the other hand, a reversed version of Proposition \ref{prop_f_r_f} can be found for sectorial matrices, as follows.
\begin{proposition}\label{prop_f_real_sec_f} Let $ f\in \mathfrak{m}$ and $ A\in\mathcal{M}_n $ be accretive such that $ W(A)\subset S_{\alpha}, $ for some $0\leq \alpha<\frac{\pi}{2}$. Then 
\begin{equation}
\Re (f(A))\leq\ \sec^{2}(\alpha) \;f(\Re A)
\end{equation}
\end{proposition}
\begin{proof}
Using Lemma \ref{lemma_reverse_real_!t}, we easily obtain
\begin{align*}
\Re f(A)=\int_{0}^{1}\Re(\mathcal{I}!_{t} A) \;d\nu_f(t)\leq \sec^{2}(\alpha)\;\int_{0}^{1}\mathcal{I}!_{t}(\Re A)\;d\nu_f(t)\leq \sec^{2}(\alpha)\;  f(\Re A),
\end{align*}
which completes the proof.
\end{proof}

Now we are ready to present the first Choi-Davis inequality for accretive matrices extending \eqref{chois_inequality_intro}. 
\begin{theorem}\label{choi_davis_accretive}
Let $ f\in \mathfrak{m} $ , $ \Phi $ be a unital  positive linear map  and $ A\in\mathcal{M}_n $ be an accretive matrix, with $W(A)\subset S_{\alpha}, 0\leq \alpha<\frac{\pi}{2}$. Then the following versions of the Choi-Davis inequality hold
$$\Re f(\Phi(A))\geq\cos^2(\alpha)\;\Re\Phi(f(A)).$$ 
\end{theorem}
\begin{proof}
 Let $A$ be an accretive matrix. Notice that 
\begin{align*}
\Re f(\Phi(A))&=\Re \int_{0}^{1}\mathcal{I}!_t\Phi(A)\;d\nu_f(t)\hspace{0.4cm}({\text{by Theorem}}\;\ref{thm_main_1})\\
&\geq \int_{0}^{1}\mathcal{I}!_t\Re(\Phi(A))\;d\nu_f(t)\hspace{0.4cm} ({\text{by Lemma}}\;\ref{2})\\
&=\int_{0}^{1}\mathcal{I}!_t\Phi(\Re A)\;d\nu_f(t)\hspace{0.4cm} ({\text{by}}\;\eqref{real_phi_real})\\
&=f(\Phi(\Re A)) \hspace{1.5cm}({\text{by Theorem}}\;\ref{thm_main_1})\\
&\geq \Phi(f(\Re A))\hspace{2cm} ({\text{by}}\;\eqref{chois_inequality_intro})\\
&\geq \cos^2(\alpha)\;\Re\Phi(f(A)),\hspace{0.4cm} ({\text{by Proposition}}\;\ref{prop_f_real_sec_f})\\
\end{align*}
which completes the proof.
\end{proof}

As an application of Theorem \ref{choi_davis_accretive}, we present the following accretive version of Lemma \ref{f<>}.

\begin{corollary}\label{theorem_inner_product} Let $A\in\mathcal{M}_n$ be an accretive matrix and $f\in\mathfrak{m}$ such that $W(A)\subset S_\alpha$. Then for any unit vector $x\in\mathbb{C}^n,$
\begin{align}
\Re\left\langle f(A)x,x\right\rangle \leq \sec^2(\alpha) \Re f\left( \left\langle Ax,x\right\rangle \right) .
\end{align}

\end{corollary} 
\begin{proof}  Letting  $\Phi(A)=\left\langle Ax,x\right\rangle$ in Theorem \ref{choi_davis_accretive}, $\Phi$ is a unital positive linear map. Then we have 
$$\cos^2(\alpha)\;\Re\Phi(f(A))\leq\Re f(\Phi(A))\Rightarrow \Re\left\langle f(A)x,x\right\rangle \leq\sec^2(\alpha)\Re f\left(\left\langle Ax,x\right\rangle\right),$$ 
which completes the proof.
\end{proof}

Recall that when $f\in\mathfrak{m}$ and $A, B\in\mathcal{M}_n^+$, then 
\begin{align}
f(A\nabla_t B)\geq f(A)\nabla_t f(B), 0<t<1.\label{concave_t}
\end{align}

Next, we present the sectorial version of \eqref{concave_t}. 
\begin{theorem} \label{thm_concave_t} Let $ A, B\in\mathcal{M}_n $ be accretive matrices such that $W(A), W(B)\subset S_{\alpha}$ for some $0\leq \alpha<\frac{\pi}{2}$. Then for any $f\in \mathfrak{m}$ and $0<t<1,$
\begin{align}
\Re(f(A)\nabla_t f(B))\leq \sec^2(\alpha) \Re f(A\nabla_t B) .\label{f_nabla}
\end{align}

\end{theorem}
\begin{proof} We have
\begin{align*}
\Re f(A\nabla_t B)&=\Re f((1-t)A+tB)\\
&\geq f((1-t)\Re A+t\Re B)\hspace{5cm}(\text{by Propostion}\; \ref{prop_f_r_f})\\
&\geq(1-t)f(\Re A)+tf(\Re B)  \hspace{5cm}(\text{by Propostion}\; \ref{oper_intro_prop}) \\
&\geq(1-t)\cos^2(\alpha)\Re f(A)+t\cos^2(\alpha)\Re f(B)\hspace{1.7cm}(\text{by Propostion}\;\ref{prop_f_real_sec_f})\\
&=\cos^2(\alpha)\Re\left((1-t)f(A)+tf(B)\right),
\end{align*}
hence
$$\Re f(A\nabla_t B)\geq \cos^2(\alpha)\Re(f(A)\nabla_t f(B)).$$
\end{proof}

 Now we present the sectorial version of Lemma \ref{lemma_ando_hiai}.
\begin{theorem}\label{theorem_sharp_nabla} Let $ A, B\in\mathcal{M}_n $ be accretive matrices such that $W(A),W(B)\subset S_{\alpha}$ and let $ f\in \mathfrak{m} $. Then
\begin{equation}
\Re(f(A)\sharp f(B)) \leq\sec^{4}(\alpha)\Re(f(A\nabla B)).
\end{equation}
\end{theorem}
\begin{proof} We have
\begin{eqnarray}
\cos^{2}(\alpha)\Re(f(A)\sharp f(B))&\leq &\Re f(A)\sharp \Re f(B)\hspace{4cm}({\text{by Lemma}}\;\ref{lemma_sharpsec})\nonumber\\
&\leq &  \left\{\sec^{2}(\alpha)f(\Re A)\right\}\sharp  \left\{\sec^{2}(\alpha)f(\Re B)\right\}\hspace{0.4cm}({\text{by Proposition}}\;\ref{prop_f_real_sec_f})\nonumber\\
&=&\sec^{2}(\alpha)f(\Re A)\sharp f(\Re B)\nonumber\\
&\leq &\sec^{2}(\alpha)f(\Re A\nabla\Re B)\hspace{3.7cm}({\text{by Lemma}}\;\ref{lemma_ando_hiai})\nonumber\\
&\leq &\sec^{2}(\alpha)\Re f(A\nabla B)\hspace{3.7cm}({\text{by Proposition}}\;\ref{prop_f_r_f}).\nonumber
\end{eqnarray}
Thus, we have shown that
\[\Re(f(A)\sharp f(B)) \leq\sec^{4}(\alpha)\Re(f(A\nabla B)),\]
which completes the proof.
\end{proof}
\section{Norm inequalities for accretive matrices}
In this section, we present norm inequalities for accretive matrices. We begin with the following accretive version of Lemma \ref{lemma_f_of_norm_f}.
\begin{proposition}\label{prop_f_norm_of_f} Let $ A\in\mathcal{M}_n $ be an accretive matrix and let $ \parallel.\parallel $ be normalized unitarily invariant norm . Then for $ f\in \mathfrak{m} $,
\begin{equation}
f(\parallel\Re A\parallel)\leq\parallel\Re f(A)\parallel.
\end{equation}
\end{proposition}
\begin{proof}
 For accretive $A$, we have
\begin{align*}
\parallel\Re f(A)\parallel &\geq \parallel f(\Re A)\parallel\hspace{3cm}(\text{by Proposition}\; \ref{prop_f_r_f})\\
&\geq f(\parallel\Re A\parallel)\hspace{3cm}(\text{by Lemma}\; \ref{lemma_f_of_norm_f}),
\end{align*}
completing the proof.
\end{proof} 
A reversed version can be stated as follows, when sectorial matrices interfere.
\begin{corollary} Let $ A\in\mathcal{M}_n $ be an accretive matrix such that $W(A)\subset S_\alpha$ for some $0\leq \alpha<\frac{\pi}{2}$ and let $ \parallel.\parallel_\infty $ be the usual operator norm. Then for $ f\in \mathfrak{m} ,$
\begin{align*}
f(\parallel\Re A\parallel_\infty)\leq\parallel\Re f(A)\parallel_\infty\leq\sec^2(\alpha)f(\parallel\Re A\parallel_\infty).
\end{align*}
\end{corollary}
\begin{proof} The first inequality follows from Proposition \ref{prop_f_norm_of_f}. For the second inequality, Proposition \ref{prop_f_real_sec_f} and Lemma \ref{f<>} imply
$$\Re\left<f(A)x,x\right>=\left\langle\Re f(A)x,x\right\rangle\leq\sec^2(\alpha)\left\langle f(\Re A)x,x\right\rangle\leq\sec^2(\alpha) f(\Re\left<Ax,x\right>). $$

Notice that since $A$ is accretive matrix, $f(A)$ is accretive by Proposition \ref{prop_f_r_f}. Taking the supremum over $\|x\|=1$ of the latter inequality implies

\begin{align*}
\|\Re f(A)\|_\infty&=\sup_{\|x\|=1}\left<\Re f(A)x,x\right>\hspace{0.4cm}(\text{since}\;f(A)\;{\text{is accretive}})\\
&\leq \sec^2(\alpha)\sup_{\|x\|=1}f(\Re\left<Ax,x\right>)\\
&=\sec^2(\alpha)f\left(\sup_{\|x\|=1}\left<\Re Ax,x\right>\right)\hspace{0.4cm}(\text{since}\;f\;{\text{is increasing}})\\
&=\sec^2(\alpha)f(\|\Re A\|_\infty)\hspace{0.4cm}(\text{since}\;\Re A>0).
\end{align*}
This completes the proof.

\end{proof}

\begin{corollary}
Let $ A, B \in \mathcal{M}_{n} $ be accretive matrices such that $ W(A), W(B)\subset S_{\alpha} $ for some $0\leq \alpha<\frac{\pi}{2}$ and let $\Phi$ be a unital positive linear map. Then, for any unitarily invariant norm $ \parallel\cdot\parallel$ and any $f\in \mathfrak{m}$,
$\displaystyle$
\begin{center}
$ \cos^{3}(\alpha)\parallel\Phi(A\sigma_f B)\parallel\ \leq
 \parallel\Phi(A)\sigma_f \Phi(B)\parallel.$
\end{center}
\end{corollary}
\begin{proof} By applying Lemma \ref{norm} and through the inequality \eqref{Phicos}, we get
\begin{align*}
\cos^{3}(\alpha)\parallel\Phi(A\sigma_f B)\parallel\ \leq\cos^{2}(\alpha)\parallel \Re\Phi(A\sigma_f B)\parallel\leq \parallel \Re\left(\Phi(A)\sigma_f \Phi(B)\right)\parallel
\leq \parallel\Phi(A)\sigma_f \Phi(B)\parallel,
\end{align*}
which completes the proof.
\end{proof}

\begin{corollary}
Let $ A, B \in \mathcal{M}_{n} $ be accretive matrices such that $ W(A), W(B)\subset S_{\alpha} $ for some $0\leq \alpha<\frac{\pi}{2}$ and let $\Phi$ be a positive linear map. Then, for any unitarily invariant norm $ \parallel\cdot\parallel$ and any $f\in \mathfrak{m}$, we have for  $t=f'(1)\in(0,1)$
\begin{center}
$ \cos^{3}(\alpha)\parallel\Phi(A\sigma_f B)\parallel\ \leq
 \parallel\Phi(A)\nabla_t \Phi(B)\parallel $
\end{center}
\end{corollary}
\begin{proof}Using Lemma \ref{norm} and by \eqref{sigmanabla} , we get
\begin{align*}
\cos^{3}(\alpha)\parallel\Phi(A\sigma_f B)\parallel\leq\cos^{2}(\alpha)\parallel \Re\Phi(A\sigma_f B)\parallel \leq\parallel \Re\left(\Phi(A)\nabla_t\Phi(B)\right)\parallel
\leq\parallel\Phi(A)\nabla_t\Phi(B)\parallel,
\end{align*}
which completes the proof.
\end{proof}
Now we are ready to present the accretive version of Lemma \ref{lemma_ando_zhan}.
\begin{theorem}\label{ando_zhan}
Let $ A, B \in \mathcal{M}_{n} $ be accretive matrices such that $ W(A), W(B)\subset S_{\alpha}$ for some $ 0\leq \alpha<\frac{\pi}{2} $. Then, for any unitarily invariant norm $ \parallel\cdot\parallel$ and any $f\in \mathfrak{m}$,
\begin{center}
$\parallel f(A+B)\parallel\leq\sec^{3}(\alpha)\parallel f(A)+f(B)\parallel.$
\end{center}
\end{theorem}
\begin{proof} We have
\begin{align*}
\cos(\alpha)\parallel f(A+B)\parallel&\leq\parallel \Re f(A+B)\parallel\hspace{2 cm}({\text{by Lemma}}\;\ref{norm})\\
&\leq\sec^{2}(\alpha)\parallel f(\Re A+\Re B)\parallel\hspace{0.4 cm}({\text{by Proposition}}\;\ref{prop_f_real_sec_f})\\
&\leq\sec^{2}(\alpha)\parallel f(\Re A)+f(\Re B)\parallel\hspace{0.4cm}({\text{by}}\;\eqref{17})\\
&\leq\sec^{2}(\alpha)\parallel \Re f(A)+\Re f(B)\parallel \hspace{0.4cm}({\text{by Proposition}}\;\ref{prop_f_r_f})\\
&=\sec^{2}(\alpha)\parallel \Re (f(A)+f(B))\parallel\\
&\leq\sec^{2}(\alpha)\parallel f(A)+f(B)\parallel \hspace{1.5cm}({\text{by Lemma}}\;\ref{norm}).
\end{align*}
Consequently
\begin{align*}
\parallel f(A+B)\parallel\leq\sec^{3}(\alpha)\parallel f(A)+f(B)\parallel,
\end{align*}
which completes the proof.
\end{proof}

The norm version of Theorem \ref{thm_concave_t} reads as follows.

\begin{corollary} Let $ A, B \in \mathcal{M}_{n} $ be accretive matrices such that $ W(A), W(B)\subset S_{\alpha} $ for some $0\leq \alpha<\frac{\pi}{2}$ and let $f\in \mathfrak{m}$. If $t\in (0,1),$ then 
\begin{align*}
\parallel f(A\nabla_t B)\parallel\geq \cos^3(\alpha)\parallel f(A)\nabla_t f(B)\parallel
\end{align*}
for any unitarily invariant norm $\|\cdot\|.$
\end{corollary}
\begin{proof}  Using Lemma \ref{norm} and by \eqref{f_nabla}, we get
$$\parallel f(A\nabla_t B)\parallel\geq \parallel\Re f(A\nabla_t B)\parallel\geq\cos^2(\alpha)\parallel \Re\left(f(A)\nabla_t f(B)\right)\parallel\geq\cos^3(\alpha)\parallel f(A)\nabla_t f(B)\parallel.$$
\end{proof}
Next, we present the accretive version of Lemma \ref{sigma_norm}.
\begin{theorem}\label{theorem_norm_of_sigma}
Let $ A, B \in \mathcal{M}_{n} $ be accretive matrices such that $ W(A), W(B)\subset S_{\alpha}$ for some $ 0\leq \alpha<\frac{\pi}{2} $. Then for any unitarily invariant norm $ \parallel\cdot\parallel$ and any $f\in \mathfrak{m}$,
\begin{align}
\parallel A \sigma_f B\parallel\leq\sec^3(\alpha)\left( \parallel A\parallel\sigma_f\parallel B\parallel\right).
\end{align}
\end{theorem}
\begin{proof} Noting Lemma \ref{norm}, \eqref{rleqsec} and Lemma \ref{sigma_norm}, we have
\begin{align*}
\parallel A \sigma_f B\parallel&\leq\sec(\alpha)\parallel\Re( A \sigma_f B)\parallel\leq\sec^3(\alpha)\parallel\Re A\;\sigma_f\;\Re B\parallel\\&\leq\sec^3(\alpha)\left(\parallel \Re A\parallel\sigma_f\parallel\Re B\parallel\right)\leq\sec^3(\alpha)\left( \parallel A\parallel\sigma_f\parallel B\parallel\right),
\end{align*}
which completes the proof.
\end{proof}

{\tiny \vskip 1 true cm }
{\tiny (Y. Bedrani) Department of Mathematics, University of Jordan, Amman 11940, Jordan. 
\textit{E-mail address:} yacinebedrani9@gmail.com}
{\tiny \vskip 0.3 true cm }
 {\tiny (F. Kittaneh) Department of Mathematics, University of Jordan, Amman 11940, Jordan. 
  \textit{E-mail address:} \textit{E-mail address:} fkitt@ju.edu.jo}
 {\tiny \vskip 0.3 true cm }
 {\tiny (M. Sababheh) Dept. of Basic Sciences, Princess Sumaya University for Tech., Amman 11941, Jordan. 
\textit{E-mail address:} sababheh@psut.edu.jo}

 {\tiny \vskip 0.3 true cm }

\end{document}